\documentclass[reqno]{amsart}

\usepackage[danish,english]{babel} 
\usepackage[latin1]{inputenc} 
\usepackage[T1]{fontenc} 
\usepackage{lmodern}
\usepackage{amsmath,amsthm,amssymb,amsfonts,mathrsfs,latexsym} 
\usepackage{graphicx} 
\usepackage{verbatim} 
\usepackage[all]{xy} 
\usepackage[numbers]{natbib}
\usepackage{enumerate}
\usepackage{hyperref}

\numberwithin{equation}{section}

\usepackage[shortlabels]{enumitem}
\setenumerate[0]{{\upshape(1)}}

\newtheorem{thm}{Theorem}[section]
\newtheorem{lem}[thm]{Lemma}
\newtheorem{prop}[thm]{Proposition}
\newtheorem{cor}[thm]{Corollary}
\newtheorem*{thm*}{Theorem}
\theoremstyle{definition}
\newtheorem{defi}[thm]{Definition}

\newtheorem{rem}[thm]{Remark}

\newcommand{\mc}[1]{\mathcal{#1}}

\newcommand{\mb}[1]{\mathbb{#1}}

\newcommand{\us}[1]{\upshape{#1}}

\newcommand{\C}{\mathbb{C}}
\newcommand{\N}{\mathbb{N}}

\newcommand{\R}{\mathbb{R}}
\newcommand{\Z}{\mathbb{Z}}

\newcommand{\la}{\langle}
\newcommand{\ra}{\rangle}
\renewcommand{\epsilon}{\varepsilon}
\let\tempphi\phi
\let\phi\varphi
\let\varphi\tempphi
\renewcommand{\bar}{\overline}

\newcommand{\Cs}{\ensuremath{C^*}}
\newcommand{\F}{\mb F}

\makeatletter
\def\blfootnote{\xdef\@thefnmark{}\@footnotetext}
\makeatother

\begin{document}
\selectlanguage{english} 


\title{A L\'evy-Khinchin formula for free groups}
\author{Uffe Haagerup}
\thanks{Both authors are supported by ERC Advanced Grant no.~OAFPG 247321 and the Danish National Research Foundation through the Centre for Symmetry and Deformation (DNRF92). The first author is also supported by the Danish Natural Science Research Council.}
\address{Department of Mathematical Sciences, University of Copenhagen,
\newline Universitetsparken 5, DK-2100 Copenhagen \O, Denmark}
\email{haagerup@math.ku.dk}

\author{S{\o}ren Knudby}
\address{Department of Mathematical Sciences, University of Copenhagen,
\newline Universitetsparken 5, DK-2100 Copenhagen \O, Denmark}
\email{knudby@math.ku.dk}
\date{\today}

\begin{abstract}
We find a L\'evy-Khinchin formula for radial functions on free groups. As a corollary we obtain a linear bound on the growth of radial, conditionally negative definite functions on free groups of two or more generators.
\end{abstract}

\maketitle



\section{Introduction}

\noindent A function $\phi$ defined on a group $G$ is called symmetric, if $\phi(x) = \phi(x^{-1})$ for every $x\in G$. It is an elementary application of Bochner's Theorem that a symmetric function $\phi:\Z\to\C$ is positive definite, if and only if
\begin{align}\label{eq:bochner1}
\phi(n) = \int_0^\pi \cos n\theta \ d\mu(\theta), \qquad n\in\Z,
\end{align}
for a positive, finite Borel measure $\mu$ on $[0,\pi]$. Moreover, $\mu$ is uniquely determined by $\phi$, and $\phi(0) = \mu([0,\pi])$. Similarly, it is a simple consequence of L\'evy-Khinchin's formula for real valued functions in the form of \cite[Corollary 18.20]{MR0481057} that a symmetric function $\psi:\Z\to\C$ with $\psi(0) = 0$ is conditionally negative definite, if and only if
\begin{align}
\psi(n) = cn^2 + \int_0^\pi (1 - \cos n\theta) \ d\rho(\theta), \qquad n\in\Z,
\end{align}
for a constant $c\geq 0$ and a Borel measure $\rho$ on $]0,\pi]$ for which
$$
\int_0^\pi (1 - \cos\theta) \ d\rho(\theta) < \infty.
$$
Moreover, $c$ and $\rho$ are uniquely determined by $\psi$. If we put
$$
\nu = c\delta_0 + (1-\cos\theta) \ d\rho(\theta),
$$
then $\nu$ is a finite Borel measure on $[0,\pi]$ uniquely determined by $\psi$, and
\begin{align}\label{eq:levy1}
\psi(n) = \int_0^\pi \frac{1- \cos n\theta}{1 - \cos\theta}\ d\nu(\theta), \qquad n\in\Z,
\end{align}
where the integrand for $\theta = 0$ should be replaced by
$$
n^2 = \lim_{\theta\to0} \frac{1- \cos n\theta}{1 - \cos\theta}.
$$
Moreover, $\psi(1) = \nu([0,\pi])$. Changing the variable $\theta$ to $s = \cos\theta$, the equations \eqref{eq:bochner1} and \eqref{eq:levy1} can be rewritten as
\begin{align}\label{eq:bochner2}
\phi(n) = \int_{-1}^1 T_n(s) \ d\mu'(s), \qquad n\geq 0,
\end{align}
and
\begin{align}\label{eq:levy2}
\psi(n) = \int_{-1}^1 \frac{1-T_n(s)}{1-s}\ d\nu'(s), \qquad n\geq 0,
\end{align}
where $\mu'$ and $\nu'$ are the image measures of $\mu$ and $\nu$ under the map $\theta\mapsto\cos\theta$, $(T_n)_{n=0}^\infty$ are the Chebyshev polynomials of first kind, and where the integrand in \eqref{eq:levy2} for $s = 1$ should be replaced by
$$
\lim_{s\to 1} \frac{1-T_n(s)}{1-s} = T'_n(1) = n^2.
$$
In this paper we will prove formulas analogous to \eqref{eq:bochner2} and \eqref{eq:levy2} for radial functions on free groups of two or more generators, i.e. functions on $\F_r$ ($2\leq r \leq \infty$) which only depend on the word length $|x|$ of an element $x\in\F_r$.

In \cite{MR665019}, Fig\`a-Talamanca and Picardello introduced the notion of spherical functions on the free groups $\F_r$ ($2\leq r < \infty$). They are complex, radial functions $\phi$ on $\F_r$, i.e. of the form
$$
\phi(x) = \dot\phi(|x|), \qquad x\in\F_r,
$$
for a unique function $\dot\phi:\N_0\to\C$, originally indexed by a parameter $z\in\C$. In this paper we will use $s = \dot\phi(1) = \frac{q}{q+1}(q^{-z} + q^{z-1})$ as the parameter, where $q = 2r - 1$. We will show in Section~\ref{sec:spherical} that with this parametrization the spherical functions $(\phi_s)_{s\in\C}$ are given by
$$
\dot\phi_s(n) = \left[ \frac{2}{q+1}\ T_n \left( \frac{q+1}{2\sqrt{q}}\ s \right) + \frac{q-1}{q+1}\ U_n \left( \frac{q+1}{2\sqrt{q}}\ s \right) \right] q^{-n/2}, \quad n\in\N_0,
$$
where $(T_n)_{n=0}^\infty$ and $(U_n)_{n=0}^\infty$ are the Chebyshev polynomials of first and second kind, respectively. For $r = \infty$ we follow the convention of \cite{HSS} and define the spherical functions on $\F_\infty$ by $\phi_s(x) = s^{|x|}$. We then prove in Section~\ref{sec:posdef} and \ref{sec:negdef} the following results analogous to \eqref{eq:bochner2} and \eqref{eq:levy2}. Theorem~\ref{thm:posdef} can be found several places in the literature in the case where $r$ is finite, see \cite{MR563948}, \cite{MR0338272}, \cite{MR658507}. For completeness, in Section~\ref{sec:posdef} we include a proof also for the case where $r$ is finite.

\begin{thm}\label{thm:posdef}
Let $2\leq r \leq \infty$ and let $\phi:\F_r\to\C$ be a radial function. The following are equivalent.
\begin{enumerate}
	\item The function $\phi$ is positive definite.
	\item There is a finite positive Borel measure $\mu$ on $[-1,1]$ such that
	$$
	\phi(x) = \int_{-1}^1 \phi_s(x)\ d\mu(s), \qquad x\in\F_r.
	$$
\end{enumerate}
Moreover, if {\us(2)} holds, then $\mu$ is uniquely determined by $\phi$, and $\phi(e) = \mu([-1,1])$.
\end{thm}

\begin{thm}\label{thm:negdef}
Let $2\leq r \leq \infty$ and let $\psi:\F_r\to\C$ be a radial function with $\psi(e) = 0$. The following are equivalent.
\begin{enumerate}
	\item The function $\psi$ is conditionally negative definite.
	\item There is a finite positive Borel measure $\nu$ on $[-1,1]$ such that
	$$
	\psi(x) = \int_{-1}^1 \psi_s(x) \ d\nu(s), \qquad x\in\F_r,
	$$
	where
	$$
	\psi_s(x) = \frac{1 - \phi_s(x)}{1-s}, \quad s\in\C\setminus\{1\},
	$$
	and
	$$
	\psi_1(x) = \lim_{s\to 1} \frac{1-\phi_s(x)}{1-s}.
	$$
\end{enumerate}
Moreover, if {\us(2)} holds, then $\nu$ is uniquely determined by $\psi$, and $\nu([-1,1]) = \psi(x)$, when $|x| = 1$.
\end{thm}

Note that \eqref{eq:bochner2} and \eqref{eq:levy2} can be considered as the case $q = r = 1$ of Theorem~\ref{thm:posdef} and Theorem~\ref{thm:negdef}.

In Proposition~\ref{prop:linear-bound-psi} we show that $\psi_s(x) \leq a|x|$ for some constant $a\geq 0$ which is independent of $s$. Thus, as a corollary of Theorem~\ref{thm:negdef} we obtain that every radial, conditionally negative definite map $\psi:\F_r\to\C$ with $\psi(e)=0$ satisfies the linear bound
\begin{align}\label{eq:linear-bound-psi}
\psi(x) \leq c|x|, \qquad x\in\F_r
\end{align}
for some constant $c\geq 0$. We note that $\psi$ is necessarily non-negative. The estimate \eqref{eq:linear-bound-psi} can also be found with a different proof in the unpublished lecture notes \cite[p.~91]{heidelberg} in the case where $r$ is finite.

The paper is organized as follows. In Section~\ref{sec:spherical} we recall some facts about spherical functions on free groups. Section~\ref{sec:posdef} and \ref{sec:negdef} contains the integral representation theorems of positive definite and conditionally negative definite radial functions, respectively. Lastly, in Section~\ref{sec:linearbound} we apply the integral representation from Theorem~\ref{thm:negdef} to deduce the linear bound on the growth of radial, conditionally negative definite functions. The linear bound \eqref{eq:linear-bound-psi} is actually a special case of Theorem 1.7 from \cite{K-semigroups}, but the proof given here of our special case is more illuminating (and much shorter).

\section{Spherical functions}\label{sec:spherical}

\noindent Let $r$ be a natural number with $r \geq 2$, and consider the free group $\F_r$ with $r$ generators. To ease notation in some places we let $q = 2r - 1$ (the same notation is used in \cite{MR1152801}, where it is also better justified). If we identify $\F_r$ with the vertices of its Cayley graph, then each non-trivial element $x\in\F_r$ has $q$ neighbors further from the identity and a single neighbor closer to the identity. The Cayley graph is a homogeneous tree of degree $q+1$, and the length function on $\F_r$ simply returns the distance to the neutral element. A function $\F_r\to\C$ is called \emph{radial}, if its value at $x\in\F_r$ only depends on the word length $|x|$. In other words, a function $\phi:\F_r\to\C$ is a radial, if and only if it has the form $\phi(x) = \dot\phi(|x|)$ for a unique map $\dot\phi:\N_0\to\C$.

For each $n\in\N_0$ we let $E_n = \{ x\in\F_r \mid |x| = n\}$. A simple counting argument (the one given above) shows that $|E_n| = (q+1)q^{n-1}$, when $n\geq 1$. Let $\mu_0 = \delta_e$ be the Dirac function at $e$, and let $\mu_n$ be the function on $\F_r$ with value $1/(q+1)q^{n-1}$ on words of length $n$ and zero otherwise.

Let $\C[\F_r]$ denote the group algebra (with convolution as product), and let $\mc A\subseteq \C[\F_r]$ denote the subspace consisting of (finitely supported) radial functions. Clearly $\mc A$ is the linear span of $(\mu_n)_{n=0}^\infty$. We call $\mc A$ the radial algebra. The following is well-known (see \cite{MR665019}).

\begin{lem}
For $n\geq 1$ we have
\begin{align}\label{eq:recurrence-mu}
\mu_1 * \mu_n = \frac{1}{q+1} \mu_{n-1} + \frac{q}{q+1}\mu_{n+1},
\end{align}
and thus the radial algebra $\mc A$ is the unital, commutative convolution algebra generated $\mu_1$.
\end{lem}

If we let $P_n$ be the polynomials defined recursively by
\begin{align}\label{eq:recurrence-pol}
\begin{array}{c}
P_0(x) = 1, \quad P_1(x) = x, \\
\\
P_{n+1}(x) = \frac{q+1}{q}\ xP_n(x) - \frac{1}{q}\ P_{n-1}(x),
\end{array}
\end{align}
then it is immediate from the lemma that
\begin{align}\label{eq:mu-pol}
\mu_n = P_n(\mu_1), \qquad n\in\N_0.
\end{align}

Whenever $\phi:\F_r\to\C$ is a function, we let $L\phi(x)$ denote the average value of $\phi$ over the neighbours of $x$. A computation shows that $(\mu_1*\phi)(x)$ is $L\phi(x)$, so the operator $L$ is simply left convolution by $\mu_1$. The operator $L$ is called the \emph{Laplace operator}.

We define spherical functions as the eigenfunctions of the Laplace operator.

\begin{defi}\label{defi:spherical}
With $\F_r$ the free group on $r$ generators ($2\leq r < \infty$) we say that a map $\phi:\F_r\to\C$ is \emph{spherical}, if the following two conditions hold.
\begin{enumerate}[(i)]
	\item $\phi$ is radial with $\phi(e) = 1$,
	\item $L\phi = s\phi$ for some $s\in\C$.
\end{enumerate}
The number $s\in\C$ is called the \emph{eigenvalue} of the spherical function $\phi$.
\end{defi}

If we (as always) let $\dot\phi:\N_0\to\C$ denote the unique map such that $\phi(x) = \dot\phi(|x|)$, then one can rewrite (i) and (ii) as
\begin{align}\label{eq:recurrence-phi}
\begin{array}{c}
\dot\phi(0) = 1, \quad \dot\phi(1) = s, \\
\\
\dot\phi(n+1) = \frac{q+1}{q}\ s\dot\phi(n) - \frac{1}{q}\ \dot\phi(n-1).
\end{array}
\end{align}

It follows that to each $s\in\C$ there is a unique spherical function with eigenvalue $s$, and the eigenvalue is determined by the function's value on words of length one. We denote this function $\phi_s$. Note that in \cite{MR1152801}, \cite{MR665019} and \cite{HSS} our $\phi_s$ is denoted $\varphi_z$, where $z\in\C$ can be any complex number satisfying $s = \frac{q}{q+1}(q^{-z} + q^{z-1})$. From \eqref{eq:recurrence-pol} and \eqref{eq:recurrence-phi} we immediately see that
\begin{align}\label{eq:phi-pol}
\dot\phi_s(n) = P_n(s).
\end{align}
for each $n\in\N_0$ and $s\in\C$.

Recall that the Chebyshev polynomials of first and second kind are defined recursively by
\begin{align*}
T_0(x) &= 1, \quad & U_0(x) &=1, \\
T_1(x) &= x, \quad & U_1(x) &=2x, \\
T_{n+1}(x) &= 2xT_n(x) - T_{n-1}(x), \quad & U_{n+1}(x) &= 2xU_n(x) - U_{n-1}(x).
\end{align*}

Using the recurrence relation \eqref{eq:recurrence-pol} one can easily show that $P_n$ can be expressed using Chebyshev polynomials. Explicitly, the following holds.
\begin{align}\label{eq:pos-spherical}
P_n(x) = \left[ \frac{2}{q+1}\ T_n \left( \frac{q+1}{2\sqrt{q}}\ x \right) + \frac{q-1}{q+1}\ U_n \left( \frac{q+1}{2\sqrt{q}}\ x \right) \right] q^{-n/2}, \qquad x\in\C
\end{align}
for each $n\in\N_0$. 

When $\phi:\F_r\to\C$ and $\psi\in\mc A$ we let
$$
\la \psi,\phi\ra = \sum_{x\in\F_r} \psi(x)\phi(x).
$$
Observe that when $\phi$ is radial, then
$$
\la \mu_n,\phi\ra = \dot\phi(n).
$$
In particular, \eqref{eq:phi-pol} shows that
\begin{align}\label{eq:pol-functional}
\la \mu_n , \phi_s \ra = P_n(s).
\end{align}

We have the following alternative characterization of spherical functions (from \cite{MR665019}).

\begin{lem}[{\cite[Lemma 2]{MR665019}}]\label{lem:sph-multiplicative}
Let $2\leq r < \infty$, and let $\phi:\F_r\to\C$ be a radial map, not identically zero. Then the following are equivalent.
\begin{enumerate}
	\item The map $\phi$ is spherical.
	\item The functional $f: \psi\mapsto \la \psi , \phi \ra$ is multiplicative on the radial algebra $\mc A$.
\end{enumerate}
\end{lem}
\begin{proof}
Suppose first $\phi$ is spherical with eigenvalue $s$. Then from \eqref{eq:mu-pol} and \eqref{eq:pol-functional} we see that
$$
\la P_n(\mu_1), \phi_s\ra = P_n(s),
$$
and since $\{P_n\}_{n=0}^\infty$ spans the set of all polynomials, we get
$$
\la P(\mu_1), \phi_s\ra = P(s), \qquad\text{for every } P\in\C[x].
$$
Since $\mu_1$ generates $\mc A$, this shows that $f$ is multiplicative on $\mc A$.

Suppose conversely that $f$ is multiplicative on $\mc A$. We will show that $\phi$ satisfies \eqref{eq:recurrence-phi} for some $s\in\C$. Since $\mu_0$ is the unit of $\mc A$, we have
$$
\la \mu_n,\phi\ra = \la \mu_0,\phi\ra\la\mu_n,\phi\ra,
$$
and beacuse $\phi$ is not identically zero, we must have $\dot\phi(0) = \la \mu_0,\phi\ra = 1$.

Let $s = \dot\phi(1) = \la \mu_1 , \phi\ra$. We have
$$
\la \mu_1 * \mu_n ,\phi \ra = \la\mu_1,\phi\ra\la\mu_n,\phi\ra = s\dot\phi(n).
$$
Also by \eqref{eq:recurrence-mu} we have
$$
\la \mu_1 * \mu_n ,\phi \ra = \frac{1}{q+1} \la \mu_{n-1},\phi\ra + \frac{q}{q+1} \la\mu_{n+1},\phi\ra = \frac{1}{q+1} \dot\phi(n-1) + \frac{q}{q+1} \dot\phi(n+1).
$$
This proves \eqref{eq:recurrence-phi} and shows that $\phi$ is spherical.
\end{proof}


To define the spherical functions on $\F_\infty$, we first define the Laplace operator $L$ on \emph{radial} functions on $\F_\infty$. When $\phi$ is a radial function on $\F_\infty$ we let $L\phi$ denote the function on $\F_\infty$ given by
$$
L\phi(x) = \dot\phi(|x| + 1), \qquad x\in\F_\infty.
$$
\begin{defi}\label{defi:spherical-inf}
As before, we say that a radial function $\phi:\F_\infty\to\C$ with $\phi(e) = 1$ is \emph{spherical}, if $L\phi = s\phi$ for some $s\in\C$ called the \emph{eigenvalue}.
\end{defi}
We see that a radial function $\phi$ is spherical with eigenvalue $s$, if and only if it satisfies the following recurrence relation.
\begin{align}\label{eq:recurrence-phi-infty}
\begin{array}{c}
\dot\phi(0) = 1, \quad \dot\phi(1) = s, \\
\\
\dot\phi(n+1) = s\dot\phi(n).
\end{array}
\end{align}
Obviously, the spherical function with eigenvalue $s\in\C$ is given by
$$
\dot\phi(n) = s^n, \qquad n\in\N_0.
$$

We will need to work with spherical functions on several free groups at once, and when this is the case we denote the spherical function on $\F_r$ with eigenvalue $s$ by $\phi_{s,r}$. Comparing \eqref{eq:recurrence-phi} with \eqref{eq:recurrence-phi-infty} we see that for each $s\in\C$ and $n\in\N_0$
\begin{align}\label{eq:ptwise-limit}
\dot\phi_{s,r}(n) \to \dot\phi_{s,\infty}(n) \qquad\text{as } r\to\infty.
\end{align}

\section{Positive definite maps}\label{sec:posdef}

\noindent In this section we prove Theorem~\ref{thm:posdef} first in the case where $2\leq r < \infty$, and then from this we deduce the case $r=\infty$ by a limiting argument. Recall that a map $\phi:G\to\C$ defined on a group $G$ is called positive definite, if
$$
\sum_{j,k=1}^n c_j\bar{c_k} \phi(x_k^{-1}x_j) \geq 0
$$
for every $n\in\N$, $\{x_1,\ldots,x_n\}\subseteq G$ and $\{c_1,\ldots,c_n\}\subseteq \C$.

In \cite[p. 53]{MR1152801} it is shown that the spherical function $\phi_s$ is positive definite, if and only if $s\in[-1,1]$. This can also be seen in a different way using the following fact, which is a consequence of The Haagerup-Paulsen-Wittstock Factorization Theorem for completely bounded maps \cite[Theorem B.7]{BO}: For a map $\phi:\F_r\to\C$ with $\phi(e) = 1$, it holds that $\phi$ is positive definite, if and only if the Herz-Schur norm of $\phi$ is 1. In \cite[Theorem 3.3]{HSS} the authors calculate the Herz-Schur norm of spherical functions, and it follows from their theorem that $\phi_s$ is positive definite, if and only if $s\in[-1,1]$.


\begin{thm}\label{thm:integral-posdef}
Let $r\geq 2$ be a natural number, and let $\phi:\F_r\to\C$ be radial with $\phi(e) = 1$. The following are equivalent.
\begin{enumerate}
	\item The map $\phi$ is positive definite.
	\item There is a probability measure $\mu$ on $[-1,1]$ such that
	\begin{align}\label{eq:int1}
	\phi(x) = \int_{-1}^1 \phi_s(x)\ d\mu(s), \qquad x\in\F_r,
	\end{align}
	where $\phi_s$ is the spherical function given by \eqref{eq:pos-spherical}.
\end{enumerate}
Moreover, if {\us(2)} holds, then $\mu$ is uniquely determined by $\phi$.
\end{thm}
\begin{proof}
\mbox{}

(2)$\implies$(1). This is trivial, since each $\phi_s$ is positive definite.

(1)$\implies$(2). Consider the universal \Cs-algebra $C^*(\F_r)$. Every state on $C^*(\F_r)$ restricts to a positive definite map $\F_r$, and by universality every positive definite map $\rho$ on $\F_r$ with $\rho(e) = 1$ extends by linearity and continuity to a state on $C^*(\F_r)$. Let $\Phi$ be the extension of $\phi$ to $C^*(\F_r)$ so that $\Phi(x) = \phi(x)$ for every $x\in\F_r \subseteq C^*(\F_r)$. We also let $\Phi_s$ denote the extension of $\phi_s$ to a state on $C^*(\F_r)$, when $s\in[-1,1]$.

Let $(\mu_n)_{n=0}^\infty$ be as above, and consider them as elements of $C^*(\F_r)$ in the usual way. The element $\mu_1$ generates the \emph{radial} \Cs-subalgebra $C^*(\mu_1)$ of $C^*(\F_r)$. By Lemma~\ref{lem:sph-multiplicative} each $\Phi_s$ is multiplicative on $C^*(\mu_1)$. Since $\mu_1 = \mu_1^*$ and $\|\mu_1\| \leq 1$, the spectrum $\sigma(\mu_1)$ of $\mu_1$ is a subset of $[-1,1]$. Conversely, by \eqref{eq:pol-functional} we see that $\Phi_s(\mu_1) = s$, so $[-1,1]\subseteq \sigma(\mu_1)$.

By spectral theory $C^*(\mu_1) \simeq C([-1,1])$. Restricting $\Phi$ to $C^*(\mu_1)$ gives a state which by the Riesz Representation Theorem has the form
\begin{align}\label{eq:riesz}
\Phi(f(\mu_1)) = \int_{-1}^1 f\ d\mu, \qquad f\in C([-1,1])
\end{align}
for a unique probability measure $\mu$ on $[-1,1]$. Since $\phi$ is radial,
$$
\dot\phi(n) = \Phi(\mu_n) = \Phi(P_n(\mu_1))
$$
and we know from \eqref{eq:riesz} and \eqref{eq:phi-pol} that
$$
\Phi(P_n(\mu_1)) = \int_{-1}^1 P_n(s)\ d\mu(s) = \int_{-1}^1 \dot\phi_s(n)\ d\mu(s).
$$
This shows
\begin{align}
\dot\phi(n) = \int_{-1}^1 \dot\phi_s(n)\ d\mu(s), \qquad n=0,1,2,\ldots
\end{align}

We now turn to prove uniqueness. Note that if $\mu$ is a probability measure satisfying \eqref{eq:int1}
then we must have
$$
\Phi(P_n(\mu_1)) = \Phi(\mu_n) = \dot\phi(n) = \int_{-1}^1 \dot\phi_s(n)\ d\mu(s) = \int_{-1}^1 P_n(s)\ d\mu(s),
$$
so it must be the unique measure guaranteed by the Riesz Representation Theorem, since the polynomials $\{P_n\}_{n=0}^\infty$ span the set of all polynomials.
\end{proof}

We wish to prove an analogue of the previous theorem for positive definite, radial maps on $\F_\infty$. Since we no longer have the aid of the radial algebra, the idea is instead to use the previous theorem together with a limit argument. But first we identify those spherical functions on $\F_\infty$ which are positive definite.

Since $\phi_{s,r}$ is positive definite, when $s\in[-1,1]$, we get by \eqref{eq:ptwise-limit} that $\phi_{s,\infty}$ is positive definite, when $s\in[-1,1]$. This may also be seen directly as follows.

When $s\in\{0,1\}$, this is easy. When $0 < s < 1$, it is a well-known result from \cite{haagerup-example} that $s\mapsto s^{|x|}$ is positive definite. When $-1 \leq s< 0$ we write $s^{|x|} = (-1)^{|x|}(-s)^{|x|}$, and since a product of positive definite maps is again positive definite, it now suffices to show that $x\mapsto (-1)^{|x|}$ is positive definite on $\F_\infty$.

Notice that the parity of $|xy|$ and $|x| + |y|$ is the same, since the reduced form of $xy$ is obtained by cancelling the same number of letters from $x$ and $y$. Hence $(-1)^{|xy|} = (-1)^{|x|+|y|}$, so $x\mapsto (-1)^{|x|}$ is a homomorphism of $\F_\infty$ into the unit circle in $\C$. It is easily seen that such homomorphisms are always positive definite.

Positive definite, radial functions are necessarily real and bounded, so when $s\notin [-1,1]$, the map $\phi_{s,\infty}$ cannot be positive definite. This shows that the positive definite spherical functions on $\F_\infty$ are precisely those with eigenvalue $s\in[-1,1]$.

\begin{lem}
For each $x\in\F_r$ the functions $s\mapsto \phi_{s,r}(x)$ converge uniformly to $s\mapsto \phi_{s,\infty}(x)$ as $r\to\infty$.
\end{lem}
\begin{proof}
We will estimate the value
$$
\delta(r,n) = \sup_{s\in[-1,1]} | \dot\phi_{s,r}(n) - \dot\phi_{s,\infty}(n) |.
$$
When $n\in\{0,1\}$, we obviously have $\delta(r,n) = 0$. When $n\geq 1$ we find using the recurrence relations \eqref{eq:recurrence-phi} and \eqref{eq:recurrence-phi-infty}
$$
| \dot\phi_{s,r}(n+1) - \dot\phi_{s,\infty}(n+1) | \leq | \dot\phi_{s,r}(n) - \dot\phi_{s,\infty}(n) | + \frac{2}{2r-1},
$$
so
$$
\delta(r,n+1) \leq
\delta(r,n) + \frac{2}{2r-1},
$$
By induction over $n$ it follows that $\delta(r,n) \to 0$ as $r\to\infty$ for each $n\in\N_0$.
\end{proof}

We are now ready to prove a version of Theorem~\ref{thm:integral-posdef}, when $r$ is infinite.
\begin{thm}\label{thm:integral-posdef-infinite}
Let $\phi:\F_\infty\to\C$ be radial with $\phi(e) = 1$. The following are equivalent.
\begin{enumerate}
	\item The map $\phi$ is positive definite.
	\item There is a probability measure $\mu$ on $[-1,1]$ such that
	\begin{align}\label{eq:int2}
	\dot\phi(n) = \int_{-1}^1 s^n\ d\mu(s), \qquad n\in\N.
	\end{align}
\end{enumerate}
Moreover, if {\us(2)} holds, then $\mu$ is uniquely determined by $\phi$.
\end{thm}
\begin{proof}
\mbox{}

(2)$\implies$(1). We have seen that the map $x\mapsto s^{|x|}$ is positive definite on $\F_\infty$ for each $s\in[-1,1]$. Hence $\phi$ is positive definite.

(1)$\implies$(2).
We consider the chain of subgroups $\F_2 \subseteq \F_3 \subseteq \ldots \subseteq \F_\infty$. Let $\phi_{s,r}$ denote the spherical function on $\F_r$ with parameter $s\in\C$. 

Suppose $\phi:\F_\infty\to\C$ is positive definite, radial and satisfies $\phi(e) = 1$. The restriction of $\phi$ to the subgroup $\F_r$ is still radial and positive definite. Using Theorem~\ref{thm:integral-posdef}, when $r$ is finite, we find a probability measure $\mu_r$ on $[-1,1]$ such that
\begin{align}\label{eq:int3}
\dot\phi(n) = \int_{-1}^1 \dot\phi_{s,r}(n)\ d\mu_r(s)
\end{align}
for all $n\in\N_0$. Let $\mu$ be any cluster point (in the vague topology) of the sequence $(\mu_r)_{r=2}^\infty$. We claim that
\begin{align}\label{eq:int4}
\dot\phi(n) = \int_{-1}^1 \dot\phi_{s,\infty}(n)\ d\mu(s), \qquad n\in\N_0.
\end{align}
Let $n\in\N_0$ and $\epsilon > 0$ be given. Since $\mu$ is a cluster point of $(\mu_r)_{r=2}^\infty$, there is some $r_0$ such that
\begin{align}\label{eq:int5}
\left| \int_{-1}^1 \dot\phi_{s,\infty}(n) \ d\mu(s) - \int_{-1}^1 \dot\phi_{s,\infty}(n) \ d\mu_r(s)  \right| < \epsilon
\end{align}
for infinitely many $r\geq r_0$. Further, since $\dot\phi_{s,r}(n)\to\dot\phi_{s,\infty}(n)$ uniformly as $r\to\infty$, there is some $r\geq r_0$ such that \eqref{eq:int5} holds and
$$
\left| \int_{-1}^1 \dot\phi_{s,\infty}(n) \ d\mu_r(s) - \int_{-1}^1 \dot\phi_{s,r}(n) \ d\mu_r(s)  \right| < \epsilon.
$$
Combining this with \eqref{eq:int3}, we obtain
$$
\left| \dot\phi(n) - \int_{-1}^1 \dot\phi_{s,\infty}(n)\ d\mu(s) \right| < 2\epsilon,
$$
which proves \eqref{eq:int4}.

Uniqueness follows in the same way as in the proof of Theorem~\ref{thm:integral-posdef}.

\end{proof}

Theorem~\ref{thm:posdef} is the combination of Theorem~\ref{thm:integral-posdef} and Theorem~\ref{thm:integral-posdef-infinite} modulo scaling of the value $\phi(e)$.

\begin{rem}
Note that (2) in Theorem~\ref{thm:integral-posdef-infinite} coincides with the characterization of positive definite functions on the semigroup $\N_0$ given in \cite[Proposition 4.9]{BCR}.
\end{rem}

\section{Conditionally negative definite maps}\label{sec:negdef}

\noindent In this section we introduce a family of conditionally negative definite maps on $\F_r$ parametrized by the interval $[-1,1]$ and prove that every radial conditionally negative definite map has an integral representation using these.

Recall that a map $\psi:G\to\C$ defined on a group $G$ is called conditionally negative definite, if $\psi$ is hermitian, i.e. $\psi(x^{-1}) = \bar{\psi(x)}$ when $x\in G$, and
$$
\sum_{j,k=1}^n c_j\bar{c_k} \psi(x_k^{-1}x_j) \leq 0
$$
for every $n\in\N$, $\{x_1,\ldots,x_n\}\subseteq G$ and $\{c_1,\ldots,c_n\}\subseteq \C$ such that $\sum_j c_j = 0$. Note that a radial, conditionally negative definite function is real-valued.

When $x\in\F_r$ we define
\begin{align}\label{eq:neg}
\begin{array}{rl}
\psi_s(x) &= \dfrac{1 - \phi_s(x)}{1-s}, \qquad -1\leq s < 1 \\
\\
\psi_1(x) &= {\displaystyle \lim_{s\to 1} \psi_s(x). }
\end{array}
\end{align}
Clearly each map $\psi_s$ is radial. Since the map $s\mapsto\dot\phi_s(n)$ is a polynomial in $s$, it is differentiable at $s = 1$, so the definition of $\psi_1$ makes sense, and $\dot\psi_1(n) = \dot\phi'_1(n)$. Note that $s\mapsto\dot\psi_s(n)$ is a polynomial in $s$ of degree $n-1$.

Each $\phi_s$ is positive definite when $s\in[-1,1]$, so $\psi_s$ is conditionally negative definite for each $s \in [-1,1[$, and since $\psi_1$ is the poinwise limit of conditionally negative maps, it is itself conditionally negative definite.

Let $s\in[-1,1]$. Since $\phi_s$ is positive definite and radial, it is real. And also $|\phi_s(x)| \leq \phi_s(e) = 1$ for each $x\in\F_r$. In particular, each $\psi_s$ is non-negative with $\psi_s(e) = 0$.

\begin{thm}\label{thm:integral-rep-psi}
Let $2 \leq r \leq \infty$, and let $\psi:\F_r\to\C$ be radial with $\psi(e) = 0$. The following are equivalent.
\begin{enumerate}
	\item The map $\psi$ is conditionally negative definite.
	\item There is a finite positive Borel measure $\nu$ on $[-1,1]$ such that
	$$
	\psi(x) = \int_{-1}^1 \psi_s(x) \ d\nu(s), \qquad x\in\F_r,
	$$
	where $\psi_s$ is the function given by \eqref{eq:neg}.
\end{enumerate}
Moreover, if {\us(2)} holds, then $\nu$ is uniquely determined by $\psi$, and $\nu([-1,1]) = \psi(x)$, when $|x| = 1$.
\end{thm}

\begin{proof}
\mbox{}

(2)$\implies$(1). This is trivial, since each $\psi_s$ is conditionally negative definite.

(1)$\implies$(2).
For $t> 0$ the map $e^{-t\psi}$ is radial, positive definite with $e^{-t\psi(e)} = 1$, so there is a probability measure $\mu_t$ on $[-1,1]$ such that
$$
e^{-t\psi(x)} = \int_{-1}^1 \phi_s(x)\ d\mu_t(s), \qquad x\in\F_r.
$$
Let $\nu_t$ be the positive measure with density $s\mapsto(1-s)/t$ with respect to $\mu_t$. Then
$$
\frac{1 - e^{-t\psi(x)}}{t} =
\int_{-1}^1 \frac{1 - \phi_s(x)}{t}\ d\mu_t(s)
= \int_{-1}^1 \psi_s(x)\ d\nu_t(s).
$$
To see the last equality, observe that both integrals may just as well be taken over the half-open interval $[-1,1[$, since
$$
\frac{1 - \phi_1(x)}{t} = 0 \quad\text{and}\quad \nu_t(\{1\}) = 0.
$$
When $|x| = 1$, we have
$$
\nu_t([-1,1]) = \frac{1- e^{-t\psi(x)}}{t} \nearrow \psi(x) \qquad\text{as } t\to 0.
$$
Hence $\nu_t$ lies in a bounded subset of $M^+([-1,1])$, the space of positive Radon measures on $[-1,1]$. By compactness there is a subnet $(\nu_{t_\alpha})$ that converges vaguely to, say $\nu$. Now,
$$
\int_{-1}^1 \psi_s(x)\ d\nu = \lim_\alpha \int_{-1}^1 \psi_s(x)\ d\nu_{t_\alpha} = \lim_\alpha \frac{1 - e^{-t_\alpha\psi(x)}}{t_\alpha} = \lim_{t\to 0} \frac{1- e^{-t\psi(x)}}{t} = \psi(x)
$$
as desired.

Uniqueness of $\nu$ follows as usual, since $\dot\psi_s(n)$ is a polynomial in $s$ of degree $n-1$.
\end{proof}

\section{Linear bound}\label{sec:linearbound}

\noindent In this section we apply the integral representation in Theorem~\ref{thm:integral-rep-psi} to give a bound on the growth of conditionally negative definite, radial functions on the free groups. The result is contained in Corollary~\ref{cor:linear-bound-psi}.

The following result about the Chebyshev polynomials will be relevant to us.
\begin{lem}
Let $P$ be a Chebyshev polynomial of either kind, and let $x_0 > 1$. For any $x\in\R$ such that $|x| < x_0$ we have
\begin{align}\label{eq:secant}
\frac{P(x_0) - P(x)}{x_0 - x} \leq P'(x_0).
\end{align}
\end{lem}
\begin{proof}
A glance at the mean value theorem shows that it suffices to prove that $P'(x) \leq P'(x_0)$ for every $x$ with $|x| < x_0$. Recall that Chebyshev polynomials are Jacobi polynomials, and since $P$ is a Jacobi polynomial, so is $P'$ up to a positive scaling factor (see (4.21.7) in \cite{MR0310533}). It is known (see p. 168 in \cite{MR0310533}) that $P'(x)$ restricted to the interval $[-1,1]$ attains its maximum at the end-point $x = 1$. We know that $P'$ is either even or odd. Finally, it is well-known that Jacobi polynomials are increasing on $[1,\infty[$, so $P'(x) \leq P'(x_0)$, when $x\in\ ]-x_0,x_0[$.
\end{proof}

\begin{lem}\label{lem:s1}
Let $2\leq r\leq \infty$. For each $s\in[-1,1]$ and $x\in\F_r$ we have
$$
\psi_s(x) \leq \psi_1(x).
$$
\end{lem}
\begin{proof}
We may of course assume that $s\neq 1$. If $r = \infty$, the result is obvious, since then $\psi_s(x) = 1 + s + s^2 + \cdots s^{|x|-1}$. So we suppose $2\leq r < \infty$. Recall that then $\dot\phi_s(n) = P_n(s)$, where $P_n$ has the form given in \eqref{eq:pos-spherical}, and that
$$
\dot\psi_s(n) = \frac{1 - P_n(s)}{1-s}.
$$
Thus, we must prove that the slope of the secant line of $P_n$ through the points $(1,P_n(1))$ and $(s,P_n(s))$ is bounded above by the slope of the tangent line of $P_n$ at $1$. To this end, let
$$
t_n(s) = T_n \left( \frac{q+1}{2\sqrt{q}}\ s \right) q^{-n/2}, \quad u_n(s) = U_n \left( \frac{q+1}{2\sqrt{q}}\ s \right) q^{-n/2},
$$
and put $\lambda = 2/(q+1)$, so that
$$
P_n(s) = \lambda t_n(s) + (1-\lambda) u_n(s).
$$
It suffices to show that 
$$
\frac{ t_n(1) - t_n(s)}{1-s} \leq t_n'(s) \quad\text{and}\quad \frac{ u_n(1) - u_n(s)}{1-s} \leq u_n'(s),
$$
when $-1 \leq s < 1$. This, in turn, is equivalent to showing
$$
\frac{ T_n(s_0) - T_n(s)}{s_0-s} \leq T_n'(s_0) \quad\text{and}\quad \frac{ U_n(s_0) - U_n(s)}{s_0-s} \leq U_n'(s_0),
$$
when $- s_0 \leq s < s_0$ and $s_0 = \frac{q+1}{2\sqrt q}$. An application of the previous lemma now completes our proof, since $s_0 > 1$.
\end{proof}

Recall the following relations concerning hyperbolic functions and Chebyshev polynomials.

\begin{align}\label{eq:chebyshev-hyperbolic}
T_n(\cosh \alpha) = \cosh(n\alpha), \quad U_n(\cosh\alpha) = \frac{\sinh((n+1)\alpha)}{\sinh\alpha}
\end{align}
for all $\alpha \neq 0$ and $n\in\N_0$.

\begin{lem}\label{lem:explicit}
Let $2\leq r < \infty$. The map $\psi_1$ has the following form.
$$
\dot\psi_1(n) = n \frac{q+1}{q-1} - \frac{2q(1-q^{-n})}{(q-1)^2} , \qquad n\in\N_0,
$$
where as usual $q = 2r -1$.

\end{lem}
\begin{proof}
If we let $\alpha = \frac12\log q$, then notice that
$$
\frac{q+1}{2\sqrt q} = \cosh\alpha, \quad \frac{q-1}{q+1} = \tanh\alpha  \quad\text{and}\quad \frac{2}{q+1} = \frac{e^{-\alpha}}{\cosh\alpha}.
$$
Then \eqref{eq:pos-spherical} takes the form
$$
P_n(s) = \left[ \frac{e^{-\alpha}}{\cosh\alpha} T_n(\cosh(\alpha)s) + \tanh(\alpha) U_n(\cosh(\alpha)s) \right] e^{-\alpha n}.
$$
Recall that $\dot\psi_1(n) = P'_n(1)$. Using the well-known facts that
$$
T'_n(x) = nU_{n-1}(x), \qquad U'_n(x) = \frac{(n+1)T_{n+1}(x) - xU_n(x)}{x^2-1},
$$
we find the following expression for $P'_n(1)$.
\small
$$
P'_n(1) = \left[ e^{-\alpha} nU_{n-1}( \cosh\alpha ) + \sinh\alpha\ \frac{(n+1)T_{n+1}(\cosh\alpha) - \cosh\alpha U_n(\cosh\alpha)}{\cosh^2\alpha - 1}       \right] e^{-\alpha n}.
$$
\normalsize
Using \eqref{eq:chebyshev-hyperbolic} we arrive after some reduction at the expression
\small
$$
\dot\psi_1(n) = \frac{e^{-\alpha n}}{\sinh\alpha} \left( n e^{-\alpha} \sinh (n\alpha) + (n+1)\cosh((n+1)\alpha) - \coth(\alpha) \sinh((n+1)\alpha) \right).
$$
\normalsize
Rewriting in terms of $q$ gives, again after some reduction,
$$
\dot\psi_1(n) = n \frac{q+1}{q-1} - \frac{2q(1-q^{-n})}{(q-1)^2}.
$$

\end{proof}

\begin{prop}\label{prop:linear-bound-psi}
Let $2\leq r \leq \infty$. There exists a constant $a\geq 0$ such that
$$
\psi_s(x) \leq a|x|
$$
for every $x\in\F_r$ and $s\in[-1,1]$. In fact, we may take $a = \frac{r}{r-1}$, when $r$ is finite, and we may take $a = 1$ when $r = \infty$.
\end{prop}
\begin{proof}
By Lemma~\ref{lem:s1} it suffices to find $a\geq 0$ such that $\psi_1(x) \leq a|x|$ for every $x\in\F_r$. It follows from Lemma~\ref{lem:explicit} that we may take $a = \frac{q+1}{q-1} = \frac{r}{r-1}$, when $r < \infty$. When $r = \infty$ we have $\psi_1(x) = |x|$, so $a = 1$ clearly works.
\end{proof}

As a corollary of Theorem~\ref{thm:integral-rep-psi} and the previous proposition we obtain the following linear bound on conditionally negative definite, radial maps.

\begin{cor}\label{cor:linear-bound-psi}
Let $2\leq r \leq \infty$, and let $\psi:\F_r\to\R$ be a radial function. Assume $\psi$ is conditionally negative definite with $\psi(e) = 0$. Then there exists a constant $c\geq 0$ such that $\psi(x) \leq c|x|$ for every $x\in\F_r$.
\end{cor}
\begin{proof}
Since $\psi$ has an integral representation as in Theorem~\ref{thm:integral-rep-psi}, it follows from Proposition~\ref{prop:linear-bound-psi} that there is a constant $c \geq 0$ such that $\psi(x) \leq c|x|$ for every $x\in\F_r$. In fact, we may take $c = \dot\psi(1) a$, where $a$ is the constant in Proposition~\ref{prop:linear-bound-psi}.
\end{proof}

\begin{rem}
In Corollary~\ref{cor:linear-bound-psi} the restriction to radial maps is essential. For instance, on the free group $\F_\infty$ there exist conditionally negative definite maps that are unbounded on the set of generators, so clearly they do not admit a bound as in Corollary~\ref{cor:linear-bound-psi}. We conclude by giving such an example.

Let $\psi:\Z\to\R$ be given by $\psi(n) = n^2$. If $\sum_j c_j = 0$, then
$$
\sum_{j,k=1}^n c_j \bar{c_k} (x_j - x_k)^2 = - 2 \left| \sum_{j=1}^n c_j x_j \right|^2 \leq 0, \qquad x_1,\ldots,x_n\in\R,
$$
so $\psi$ is conditionally negative defintie on $\Z$. Let $\{b_n\}_n$ be the generators of $\F_r$, and let $\rho:\F_r\to\Z$ be the homomorphism given by $\rho(b_n) = n$. The composition $\psi\circ\rho$ is then conditionally negative definite on $\F_r$, and
$$
(\psi\circ\rho)(b_n) = n^2
$$
Clearly, $\psi\circ\rho$ is not radial. It is also obvious that there is no $c\geq 0$ such that $(\psi\circ\rho)(x) \leq c |x|$ for every $x\in\F_r$, because $(\psi\circ\rho)(b_1^n) = n^2$ for all $n\in\N$, while $|b_1^n| = n$. If $r = \infty$, then $\psi\circ\rho$ is even unbounded on the set of generators.
\end{rem}

\section*{Acknowledgements}
We thank Marek Bo{\.z}ejko for providing us with several references to the literature. We also thank Kristian K. Olesen for useful comments.


\end{document}